\documentclass[leqno,11pt,a4paper]{amsart}
\usepackage{bw}
\begin{document}
\vs
\title[Looijenga towers and asymptotics of ECH]{Towers of Looijenga pairs and asymptotics of ECH capacities}
\maketitle

\begin{abstract}
ECH capacities are rich obstructions to symplectic embeddings in $4$-dimensions that have also been seen to arise in the context of algebraic positivity for (possibly singular) projective surfaces. We extend this connection to relate general convex toric domains on the symplectic side with towers of polarised toric surfaces on the algebraic side, and then use this perspective to show that the sub-leading asymptotics of ECH capacities for all convex and concave toric domains are $O(1)$. We obtain sufficient criteria for when the sub-leading asymptotics converge in this context, generalising results of Hutchings and of the author, and derive new obstructions to embeddings between toric domains of the same volume. We also propose two invariants to more precisely describe when convergence occurs in the toric case. Our methods are largely non-toric in nature, and apply more widely to towers of polarised Looijenga pairs.
\end{abstract}

\section{Introduction}

We outline the symplectic part of the story -- in particular, the applications to sub-leading asymptotics of ECH capacities -- in \S\ref{sec:symp_persp} and describe the novel aspects of our algebro-geometric methods and constructions in \S\ref{sec:alg_persp}.

\subsection{Symplectic perspective} \label{sec:symp_persp}

A great deal of symplectic geometry has been stimulated by symplectic embedding problems. These are problems of the form: given two symplectic manifolds $(X,\omega)$ and $(X,\omega')$ of the same dimension, when does there exist a smooth embedding $\iota\colon X\to X'$ such that $\iota^*\omega'=\omega$? Such embeddings are called \emph{symplectic embeddings}. If there is a symplectic embedding $(X,\omega)\to(X',\omega')$ we write $(X,\omega)\se(X',\omega')$.

For each symplectic embedding problem there is a `constructive' aspect in which the aim is to show the existence of a symplectic embeddings. Conversely there is an `obstructive' aspect that usually involves finding invariants that are `monotone' under symplectic embeddings and hence obstruct their existence. We will focus on the latter here.

\emph{ECH capacities} were introduced by Hutching \cite{hut_qua_11} to obstruct embeddings between symplectic $4$-manifolds. To a symplectic $4$-manifold $(X,\omega)$ ECH associates a non-decreasing sequence
$$\{\cech_k(X,\omega)\}_{k\in\Z_{\geq0}}$$
of (extended) real numbers such that
$$(X,\omega)\se(X',\omega')\Longrightarrow\cech_k(X,\omega)\leq\cech_k(X',\omega')$$
for all $k\in\Z_{\geq0}$. ECH capacities have found many applications to notable embedding problems \cite{mcd_hof_11,ccfhr_sym_14,cri_sym_19} and have been shown to have significant connections with the geometry of divisors on algebraic surfaces \cite{wor_ech_19,cw_ech_20,wor_alg_20} and the lattice combinatorics of polytopes \cite{cls_new_15,wor_ech_19,ck_ehr_20}. For a general introduction to ECH, see \cite{hut_lec_14}. One of the main contributions of this paper is to extend and refine these connections in the context of a wider class of spaces.

The algebraic analogues -- \emph{algebraic capacities} -- to ECH capacities were introduced by the author in \cite{wor_ech_19,wor_num_20} and have since been applied to embedding problems into closed surfaces \cite{cw_ech_20} and to the asymptotics of ECH capacities as $k\to\infty$ \cite{wor_alg_20,chmp_inf_20}. To a pair $(Y,A)$ consisting of a projective algebraic surface $Y$ with an ample (or big and nef) $\R$-divisor $A$ -- a \emph{polarised surface} -- the associated algebraic capacities form a non-decreasing sequence
$$\{\calg_k(Y,A)\}_{k\in\Z_{\geq0}}$$
of real numbers with certain appealing properties; some of which we will describe shortly.

As it will occupy much of this paper, we outline the current asymptotic understanding of ECH capacities and the motivation for further study. We start with the `Weyl law' for ECH.

\begin{thm*}[{\cite[Thm.~1.1]{chr_asy_15}}] Suppose $(X,\omega)$ is a Liouville domain such that all $\cech_k(X,\omega)$ are finite. Then
$$\lim_{k\to\infty}\frac{\cech_k(X_\Omega)^2}{k}=4\on{vol}(X,\omega)$$
\end{thm*}

In other words $\cech_k(X,\omega)\sim\sqrt{4\on{vol}(X,\omega)k}$. This was used by Cristofaro-Gardiner--Hutchings \cite{ch_one_16} to prove a refinement of the Weinstein conjecture. One can hence define \emph{error terms}
$$e_k(X,\omega):=\cech_k(X,\omega)-\sqrt{4\on{vol}(X,\omega)}$$
These are manifestly $o(\sqrt{k})$ but we consider what more there is to say. A series of estimates due to Sun \cite{sun_est_18}, Cristofaro-Gardiner--Savale \cite{cs_sub_18}, and Hutchings \cite{hut_ech_19} found bounds of the form $O(k^p)$ for $p=125/252,p=2/5,p=1/4$ respectively in quite high generality. There are two conjectures significantly extending these estimates that motivate much of this paper.

\begin{conjecture*} \label{conj:intro} Let $(X,\omega)$ be a star-shaped domain in $\R^4$.
\begin{enumerate}
\item \textnormal{(c.f.~\cite[Ex.~1.6]{hut_ech_19})} $e_k(X,\omega)=O(1)$.
\item \textnormal{(\cite[Conj.~1.5]{hut_ech_19})} If $(X,\omega)$ is \textnormal{generic} then $e_k(X,\omega)$ converges with limit
$$\lim_{k\to\infty}e_k(X,\omega)=-\frac{1}{2}\on{Ru}(X,\omega)$$
where $\on{Ru}(X,\omega)$ is the \textnormal{Ruelle invariant} of $(X,\omega)$.
\end{enumerate}
\end{conjecture*}

The Ruelle invariant was defined for nice star-shaped domains in $\R^4$ by Hutchings \cite[Def.~1.4]{hut_ech_19} following ideas of Ruelle \cite{rue_rot_85}, and one can extend this definition to more general star-shaped domains by continuity. Some of the results of this paper can be viewed as seeking to make precise what `generic' means.

We consider a class of toric symplectic $4$-manifolds that will play a central role for us. Suppose $\Omega\subseteq\R^2_{\geq0}$ is a simply-connected region. Define the \emph{toric domain}
$$X_\Omega:=\mu^{-1}(\Omega)$$
where $\mu\colon\C^2\to\R^2$ is the moment map for the $S^1\times S^1$-action on $\C^2$. If $\partial\Omega$ intersects the coordinate axes in sets of the form $\{(x,0):x\in[0,a]\}$ and $\{(0,y):y\in[0,b]\}$ and the remaining part of the boundary -- which we denote by $\partial^+\Omega$ -- is a convex curve (that is, the region $\Omega$ is convex), we say that $\Omega$ is a \emph{convex domain} and that $X_\Omega$ is a \emph{convex toric domain}. This class includes balls, ellipsoids, polydisks, and many other classic symplectic manifolds. We denote the quantities $a$ and $b$ appearing above by $a(\Omega)$ and $b(\Omega)$.

In \cite[Thm.~4.10]{wor_alg_20} it was shown that when $\Omega=q\Omega_0$ for some lattice polygon $\Omega_0$ and some $q\in\R_{>0}$, we have
$$e_k(X_\Omega)=O(1)$$
with an explicit calculation of the $\limsup$ and $\liminf$, which are always different and so $e_k(X_\Omega)$ does not converge in this situation. We say that such $\Omega$ are of \emph{scaled-lattice type}. The approach to proving this result uses the fact that the interior $X_\Omega^\circ$ can be realised as the complement of an ample divisor $A_\Omega$ in the (possibly singular) projective toric surface $Y_\Omega$ associated to $\Omega$ \cite[\S2.3]{cls_tor_11}, and that by \cite[Thm.~1.5]{wor_ech_19} the ECH capacities $\cech_k(_\Omega)$ are given by the algebraic capacities $\calg_k(Y_\Omega,A_\Omega)$ associated to the pair $(Y_\Omega,A_\Omega)$.

In \cite{hut_ech_19} Hutchings showed that Conj.~\ref{conj:intro}(ii) holds for $X_\Omega$ when $\Omega$ is `strictly convex'\footnote{Hutchings' result also shows that Conj.~\ref{conj:intro}(ii) is true when $\Omega$ is `strictly concave'.}: namely, has smooth boundary and all outward normals to $\partial^+\Omega$ live in the strictly positive quadrant of $\R^2$. The Ruelle invariant in this setting is equal to $a(\Omega)+b(\Omega)$.

These two cases will be subsumed in the following theorem. The \emph{affine length} $\laff(v)$ of a vector $v\in\R^2$ is the pseudonorm defined by $0$ if $qv\notin\Z^2$ for any $q\in\R_{>0}$ and by $1$ if $v$ is a primitive vector in $\Z^2$. We define the affine length of a continuous curve by the (possibly empty) sum of the affine lengths of the direction vectors defining each linear segment of the curve.

\begin{thm*}[Cor.~\ref{cor:main}] \label{thm:main_intro} Let $X_\Omega$ be a convex toric domain, then
\begin{align*}
-\frac{1}{2}\left(a(\Omega)+b(\Omega)-\frac{1}{2}\laff(\partial^+\Omega)\right)&\geq\limsup_{k\to\infty}e_k(X_\Omega) \\
&\geq\liminf_{k\to\infty}e_k(X_\Omega)\geq-\frac{1}{2}\left(a(\Omega)+b(\Omega)+\frac{1}{2}\laff(\partial^+\Omega)\right)
\end{align*}
where $\partial^+\Omega$ is the part of $\partial\Omega$ not on the coordinate axes. In particular, $e_k(X_\Omega)=O(1)$.
\end{thm*}

Notice that the upper and lower bounds can easily be translated to involve the Ruelle invariant, and that their midpoint is exactly $-\frac{1}{2}\on{Ru}(X_\Omega)$ when $X_\Omega$ is strictly convex. We immediately obtain the following corollary generalising \cite[Thm.~1.10]{hut_ech_19}.

\begin{cor*}[Cor.~\ref{cor:main}] \label{cor:intro_main} Let $X_\Omega$ be a convex toric domain. When $\partial^+\Omega$ has no rational-sloped edges we have that $e_k(X_\Omega)$ is convergent and
$$\lim_{k\to\infty}e_k(X_\Omega)=-\frac{1}{2}(a(\Omega)+b(\Omega))$$
\end{cor*}

If a region $\Delta\subseteq\R^2_{\geq0}$ is bounded above by the graph of a convex function $f\colon[0,a]\to\R_{\geq0}$ we say that $X_\Delta$ is a \emph{concave toric domain}. Via formal properties of ECH capacities we obtain an analogous result for concave toric domains.

\begin{thm*}[Thm.~\ref{cor:main_conc}] \label{thm:intro_conc} Let $X_\Delta$ be a concave toric domain. Then
\begin{align*}
-\frac{1}{2}\left(a(\Delta)+b(\Delta)-\laff(\partial^+\Delta)\right)&\geq\limsup_{k\to\infty}e_k(X_\Delta) \\
&\geq\liminf_{k\to\infty}e_k(X_\Delta)\geq-\frac{1}{2}\left(a(\Delta)+b(\Delta)+\laff(\partial^+\Delta)\right)
\end{align*}
and so $e_k(X_\Delta)=O(1)$. If $\partial^+\Delta$ has no rational-sloped edges then
$$\lim_{k\to\infty}e_k(X_\Delta)=-\frac{1}{2}(a(\Delta)+b(\Delta))$$
\end{thm*}

We note that one consequence of better understanding the asymptotics of $e_k(X,\omega)$ is to obtain finer embedding obstructions between symplectic $4$-manifolds of the same volume (far from vacuous, as discussed in \cite[Rmk.~1.14]{hut_ech_19}). As a corollary to Thm.~\ref{thm:main_intro} and Thm.~\ref{thm:intro_conc} we obtain the following embedding obstruction subsuming \cite[Cor.~1.13]{hut_ech_19} and \cite[Cor.~5.15]{wor_alg_20}. For the purposes of this result we say that a toric domain $X_\Omega$ is `admissible' if either $\Omega$ is concave or convex and $\partial^+\Omega$ has no rational-sloped edges, or if $\Omega$ is convex and of scaled-lattice type.

\begin{cor*} Let $X_\Omega$ and $X_{\Omega'}$ be admissible toric domains of the same symplectic volume. Then
$$X_\Omega^\circ\se X_{\Omega'}\Longrightarrow\laff(\partial\Omega)\geq\laff(\partial\Omega')$$
\end{cor*}

We conclude in \S\ref{sec:outlook} by discussing two invariants -- the number of rational-sloped edges in $\partial^+\Omega$ and the degree of independence over $\Q$ of their affine lengths -- that we believe might further govern the asymptotics of $e_k(X_\Omega)$ for convex and concave toric domains.

\subsection{Algebraic perspective} \label{sec:alg_persp}

Our approach to the results in \S\ref{sec:symp_persp} is to identify an algebraic object whose `algebraic capacities' agree with the ECH capacities of a convex toric domain $X_\Omega$. We will give the formal definition in \S\ref{sec:alg_cap} but, in short, one can think of algebraic capacities as positivity invariants of a polarised surface $(Y,A)$ obtained as solutions to quadratic optimisation problems on the nef cone of $Y$. We denote the $k$th algebraic capacity of $(Y,A)$ by $\calg_k(Y,A)$.

As discussed above, when $\Omega$ is a rational-sloped polygon one can recover the ECH capacities of $X_\Omega$ as the algebraic capacities of the polarised toric surface $(Y_\Omega,A_\Omega)$ corresponding to $\Omega$. When $\Omega$ is a non-polytopal convex domain, we will use the \emph{weight expansion} of $\Omega$ \cite{mcd_sym_09,ccfhr_sym_14} in \S\ref{sec:tibs} to define a tower of polarised toric surfaces
$$(Y_0,A_0)\overset{\pi_1}{\longleftarrow}(Y_1,A_1)\overset{\pi_2}{\longleftarrow}\dots\overset{\pi_n}{\longleftarrow}(Y_n,A_n)\overset{\pi_{n+1}}{\longleftarrow}\dots$$
for which there exists a notion of algebraic capacities extending the definition for polarised surfaces. We denote such towers of polarised surfaces by calligraphic letters $(\mY,\mA)$ and denote their algebraic capacities by $\calg_k(\mY,\mA)$.

\begin{prop*}[Prop.~\ref{prop:ech_alg}] Let $\Omega$ be a convex domain, and let $(\mY_\Omega,\mA_\Omega)$ denote the tower of polarised toric surfaces associated to $\Omega$. Then
$$\cech_k(X_\Omega)=\calg_k(\mY_\Omega,\mA_\Omega)$$
for all $k\in\Z_{\geq0}$.
\end{prop*}

\begin{remark*}
A point of independent interest here is that the tower of polarised toric surfaces we produce can be viewed naturally as the object in toric algebraic geometry corresponding to the convex non-polytopal region $\Omega$.
\end{remark*}

Just as in \cite{wor_alg_20} we find that our results on the asymptotics of algebraic capacities for towers of polarised toric surfaces do not require much of the toric structure and in fact apply to a much larger class of algebro-geometric objects.

We will say that a \emph{Looijenga pair} \cite{loo_rat_81,ghk_mir_15} is a pair $(Y,L)$ consisting of a $\Q$-factorial rational surface $Y$ with a singular nodal curve $L\in{|}{-}K_Y|$. Recall that $\Q$-factorial means that an integer multiple of each Weil divisor on $Y$ is Cartier. Note that elsewhere in the literature it is standard to assume that $Y$ is smooth, in which case $L$ is either an irreducible rational nodal curve or a cycle of smooth rational curves.

We consider Looijenga pairs with a polarisation supported on the anticanonical divisor, and towers of such objects in which the choices of polarisation and anticanonical divisor are respected appropriately. We call such objects \emph{polarised Looijenga towers} and write them as pairs $(\mY,\mA)$. The towers of polarised toric surfaces we consider are examples of these, though there are also many interesting non-toric examples. We develop the necessary birational geometry of polarised Looijenga towers in \S\ref{sec:div}, including a natural notion of divisor (which includes the polarisation $\mA$), an intersection pairing, and a notion of canonical divisor $K_\mY$.

\begin{thm*}[Thm.~\ref{thm:ind_weyl} + Thm.~\ref{thm:loo_main}]Let $(\mY,\mA)$ be a polarised Looijenga tower. Then $\calg_k(\mY,\mA)\sim\sqrt{2\mA^2k}$ and the error terms $\ealg_k(\mY,\mA):=\calg_k(\mY,\mA)-\sqrt{2\mA^2k}$ satisfy
\begin{align*}
\frac{1}{2}K_\mY\cdot\mA-K_\mY^+\cdot\mA&\geq\limsup_{k\to\infty}\ealg_k(\mY,\mA) \\
&\geq\liminf_{k\to\infty}\ealg_k(\mY,\mA)\geq\frac{1}{2}K_\mY\cdot\mA
\end{align*}
where $K_\mY^+$ is a divisor on $\mY$ canonically associated to $(\mY,\mA)$.
\end{thm*}

In the case that $(\mY,\mA)$ is a tower of polarised toric surfaces arising from a convex domain $\Omega$ we calculate
$$K_\mY\cdot\mA=-\laff(\partial\Omega)=-\left(a(\Omega)+b(\Omega)+\laff(\partial^+\Omega)\right)\text{ and }-K_\mY^+\cdot\mA=\laff(\partial^+\Omega)$$
establishing Thm.~\ref{thm:main_intro} and its consequences from \S\ref{sec:symp_persp}. We prove a convergence criterion similar to Cor.~\ref{cor:intro_main} in Prop.~\ref{prop:bal_conv}. Using intersection theory on $\mY$ we also formulate algebraic capacities intrinsically in terms of divisors on $\mY$ in Prop.~\ref{prop:intrinsic}. We hope that this `intrinsic' geometry of $\mY$ will shed more insight on the asymptotics of algebraic capacities and, hence, of ECH capacities.

\subsection*{Acknowledgements} I am grateful for many encouraging and helpful conversation with Dan Cristofaro-Gardiner, Michael Hutchings, Julian Chaidez, Vinicius Ramos, Tara Holm, and Ana Rita Pires. I am especially grateful to Michael Hutchings for discussing the content of \cite{hut_ech_19} with me, and to Vinicius Ramos for hosting me at IMPA where the idea for this project was seeded. I am very thankful to \DJ an-Daniel Erdmann-Pham for providing the proof of Lemma \ref{lem:analysis}.

\section{Towers of Looijenga pairs}

\subsection{Looijenga pairs and Looijenga towers} \label{sec:loo}

In our context we will define a \emph{Looijenga pair} to be a pair $(Y,L)$ consisting of:
\begin{itemize}
\item a $\Q$-factorial rational surface $Y$,
\item a singular nodal curve $L\in{|}{-}K_Y|$.
\end{itemize}
The basic example of a Looijenga pair is a toric surface equipped with the union of its torus-invariant divisors. Note that Looijenga pairs are usually assumed to be smooth elsewhere in the literature.

A \emph{polarised Looijenga pair} is a triple $(Y,L,A)$ consisting of a Looijenga pair and an ample divisor supported on a subset of $L$. This implies that the Looijenga pair is `positive' in the language of \cite{ghk_mir_15}. If $A$ is only big and nef we say that $(Y,L,A)$ is a \emph{pseudo-polarised Looijenga pair}.

A \emph{toric blowup} of a Looijenga pair $(Y,L)$ is a blowup $\pi\colon\wt{Y}\to Y$ with centre a node of $L$. Observe that in this case the divisor $\wt{L}+E$ -- the strict transform of $L$ plus the exceptional divisor -- is such that $(\wt{Y},\wt{L}+E)$ is a Looijenga pair. We will consider towers
$$\mY:(Y_0,L_0)\overset{\pi_1}{\longleftarrow} (Y_1,L_1)\overset{\pi_2}{\longleftarrow}\dots\overset{
\pi_{n-1}}{\longleftarrow}(Y_{n-1},L_{n-1})\overset{\pi_n}{\longleftarrow}(Y_n,L_n)\overset{\pi_{n+1}}{\longleftarrow}\dots$$
of Looijenga pairs where each map $\pi_n$ is a toric blowup. We call such structures \emph{Looijenga towers}.

We can also ask that each Looijenga pair is polarised and that the toric blowups are compatible with the polarisations. Namely, we want to consider towers
$$(Y_0,L_0,A_0)\overset{\pi_1}{\longleftarrow} (Y_1,L_1,A_1)\overset{\pi_2}{\longleftarrow}\dots\overset{
\pi_{n-1}}{\longleftarrow}(Y_{n-1},L_{n-1},A_{n-1})\overset{\pi_n}{\longleftarrow}(Y_n,L_n,A_n)\overset{\pi_{n+1}}{\longleftarrow}\dots$$
of polarised Looijenga pairs where each map $\pi_n$ is a toric blowup, and the polarisations are related by
$$A_n=\pi_n^*A_{n-1}-a_nE_n$$
for some $a_n>0$, where $E_n$ is the exceptional fibre of $\pi_n$. We call such a structure a \emph{polarised Looijenga tower} and denote it by a pair $(\mY,\mA)$ where $\mY$ is the underlying Looijenga tower and $\mA$ is the sequence of polarisations $(A_n)_{n\in\Z_{\geq0}}$. If the polarisations are relaxed to pseudo-polarisations, we call $(\mY,\mA)$ a \emph{pseudo-polarised Looijenga tower}. We say that $\mY$ is smooth if $Y_0$ is smooth.

\begin{lemma} Let $(\mY,\mA)$ be a pseudo-polarised Looijenga tower. Then $\lim_{n\to\infty}A_n^2$ exists.
\end{lemma}

This also holds if we omit the anticanonical divisor $L_n$ and only consider a tower of pseudo-polarised surfaces related by arbitrary blowups.

\begin{proof} We see that
$$A_n^2=(\pi_n^*A_{n-1}-a_nE_n)^2=A_{n-1}^2+a_n^2E_n^2\leq A_{n-1}^2$$
and so $A_n^2$ is a decreasing sequence. It is bounded below since $A_n^2\geq0$ for all $n$ by the assumption that each $(Y_n,A_n)$ is pseudo-polarised and is hence convergent.
\end{proof}

As a result we define
$$\mA^2:=\lim_{n\to\infty}A_n^2=\inf\{A_n^2:n\in\Z_{\geq0}\}$$

\subsection{Polarised Looijenga towers from weighted posets} \label{sec:wt_poset}

We generalise the previous construction to use a poset other than $\Z_{\geq0}$ to index toric blowups. The towers that come from this construction can thus have many spires. We start by constructing the universal Looijenga tower $\mY_{(Y,L)}^\text{univ}$ associated to a Looijenga pair $(Y,L)$ and realise all pseudo-polarised Looijenga towers with $(Y_0,L_0)=(Y,L)$ in terms of it.

Let $(Y,L)$ be a Looijenga pair. We construct a poset $\mP_{(Y,L)}$ as follows. Let $\mP_0$ be the poset consisting of all nodes of $D$ with no order relations. Let $\pi_p\colon Y_p\to Y$ denote the toric blowup at a node $p\in L$ with exceptional divisor $E_p$. Set $\mP_p=\{p_1,p_2\}\cup\mP_0$ where $p_1,p_2$ are the two intersection points of $E_p$ with the strict transform of $L$. Define
$$\mP_1=\bigcup_{p\in\mP_0}\mP_p$$
and view this as poset by setting $q\leq p$ if and only if $q\in E_p$. Note that the elements of $\mP_1\setminus\mP_0$ correspond to nodes on the Looijenga pair $(Y_1,L_1)$ obtained from $(Y,L)$ by blowing up all the nodes of $D$. Repeating this process by blowing up each node on $L_1$ produces a new poset $\mP_2$ such that $\mP_2\setminus\mP_1$ is the set of nodes of the Looijenga pair $(Y_2,L_2)$ obtained by blowing up all nodes of $(Y_1,L_1)$. Continuing this procedure defines a Looijenga pair $(Y_n,A_n)$ for each $n\in\Z_{\geq0}$ -- letting $(Y,L)=(Y_0,L_0)$ -- and a poset $\mP_n$ such that $\mP_n\setminus\mP_{n-1}$ is the set of nodes of $(Y_n,L_n)$. These pairs coalesce to form a slightly more general kind of tower where blowups with multiple centres are permitted at each stage. For the remainder of this section we will use the term `Looijenga tower' to include such towers.

We call the Looijenga tower arising from this construction the \emph{universal Looijenga tower associated to $(Y,L)$} and denote it by $\mY_{(Y,L)}^\text{univ}$. We will later compare this to a construction of Hutchings \cite[\S3]{hut_ech_19}. Define the poset
$$\mP_{(Y,L)}=\bigcup_{n\geq0}\mP_n$$
Observe that $\mP_{(Y,L)}$ is a graded poset with grading defined by the filtration $\mP_n$. For $q\in\mP_{(Y,L)}$ of degree $n$ we obtain a Looijenga pair $(Y_q,L_q)$ obtained as the toric blowup of $(Y_n,L_n)$ at $q$.

\begin{definition}
Let $\mP$ be a countable poset. We call a function $\on{wt}\colon\mP\to\R_{\geq0}$ a \emph{weight function} if
\begin{itemize}
\item $\on{wt}$ is a poset homomorphism where $\R_{\geq0}$ is regarded as a poset in the usual way,
\item $\sum_{p\in\mP}\on{wt}(p)<\infty$.
\end{itemize}
A pair $(\mP,\on{wt})$ of a poset with a weight function is called a \emph{weighted poset}. We define the \emph{weight sequence} $\on{wt}(\mP)$ associated to a weighted poset $(\mP,\on{wt})$ to be the multiset $\{\on{wt}(p):p\in\mP\}$.
\end{definition}

We also write $\on{wt}(\mY,\mA):=\on{wt}(\mP_{(Y_0,L_0)})$ and refer to this as the weight sequence of $(\mY,\mA)$. We say that an element $q$ of a poset $\mP$ is a \emph{direct descendant} of $p\in\mP$ if $p>q$ and there is no $r\in\mP$ such that $p>r>q$.

From the data of a pseudo-polarised Looijenga pair $(Y,L,A)$ and a weight function satisfying some conditions on $\mP_{(Y,L)}$ we can produce a pseudo-polarised Looijenga tower $(\mY,\mA)$. Let $\on{wt}$ be a weight function on $\mP_{(Y,L)}$. To define a polarisation on each $(Y_n,L_n)$ we start by setting
$$A_1=\pi_1^*A_0-\sum_{p\in\mP_0}\on{wt}(p)E_p$$
and then recurse by setting
$$A_n=\pi_n^*A_{n-1}-\sum_{p\in\mP_n\setminus\mP_{n-1}}\on{wt}(p)E_p$$
If this recipe defines a polarisation (resp.~pseudo-polarisation) on each $(Y_n,L_n)$ then we say that $\on{wt}$ is an \emph{ample (resp.~big and nef)} weight function on $\mP_{(Y,L)}$. Thus, after choosing a big and nef weight function, to each $q\in\mP_{(Y,L)}$ there is a pseudo-polarised Looijenga pair $(Y_q,L_q,A_q)$. It follows by direct computation that $\on{wt}$ is big and nef implies $\on{wt}$ is a poset homomorphism.

We can non-canonically create a pseudo-polarised Looijenga tower in which each map $\pi_n$ is a single toric blowup as in \S\ref{sec:loo} from this data. We choose a bijective poset homomorphism $h\colon\mP_{(Y,L)}\to\Z_{\geq0}^\text{op}$, where $\Z_{\geq0}^\text{op}$ is $\Z$ with reverse ordering, and define $(Y_n,L_n,A_n)$ to be the pseudo-polarised Looijenga pair obtained by blowing up in the nodes $h^{-1}\{0,\dots,n\}$. One can easily verify that this is well-defined by the requirement that $h$ is a poset homomorphism.

In later sections we will choose $h$ such that $\on{wt}(h^{-1}(n))\geq\on{wt}(h^{-1}(n+1))$; in other words, there is a commutative diagram in the category of posets of the form:
$$\xymatrix{\mP_{(Y,L)} \ar[rr]^-{h} \ar[rd]_-{\on{wt}} & & \Z_{\geq0}^\text{op} \ar@{-->}[ld] \\
& \R_{\geq0}}$$
Of course $h$ is not a poset isomorphism in general!

We view two pseudo-polarised Looijenga pairs $(Y,L,A)$ and $(Y',L',A')$ as `equivalent' if there is a Looijenga pair $(Y'',L'')$ and two maps $\pi\colon Y''\to Y$ and $\pi'\colon Y''\to Y'$ given as compositions of toric blowups such that
$$\pi^*A=(\pi')^*A'$$
In this way, one can indeed recover any pseudo-polarised Looijenga tower $(\mY,\mA)$ with $(Y_0,L_0)=(Y,L)$ up to equivalence from $\mY_{(Y,L)}^\text{univ}$ by assigning a weight of zero to all nodes on toric blowups of $(Y,L)$ that are not blown up in $\mY$. We will revisit this notion of equivalence in \S\ref{sec:alg_cap}.

\subsection{Divisors on Looijenga towers} \label{sec:div}

Throughout this subsection we fix a Looijenga tower $\mY=\{(Y_n,L_n)\}$ with toric blowup maps $\pi_n$ and exceptional divisors $E_n$. We will introduce the notion of divisors on $\mY$, and study classes of divisors that will be relevant to our applications. Let $\mathbb{K}\in\{\Z,\Q,\R\}$.

\begin{definition} A \emph{$\mathbb{K}$-divisor} on $\mY$ is a sequence $\mD=\{D_n\}$ where $D_n$ is a $\mathbb{K}$-divisor on $Y_n$ such that
$$D_n=\pi_n^*D_{n-1}-d_nE_n$$
for some $d_n\in\mathbb{K}$.
\end{definition}

Clearly one can view a polarisation $\mA$ on $\mY$ as an $\R$-divisor on $\mY$. We call the sequence $(d_n)_{n\in\Z_{\geq1}}$ the \emph{weight sequence} of $\mD$. The weight sequence of $\mA$ regarded as a divisor is by construction the weight sequence of $(\mY,\mA)$ as defined in \S\ref{sec:wt_poset}.

$\mY$ has a \emph{canonical divisor} $K_\mY$ defined as the sequence
$$K_\mY=\{K_{Y_n}\}_{n\in\Z_{\geq0}}$$
When $\mY$ is smooth the weight sequence of $K_\mY$ is $(1,1,\dots)$. We denote the set of $\mathbb{K}$-divisors on $\mY$ by $\on{Div}(\mY)_\mathbb{K}$. One can easily modify this definition to produce numerical or linear equivalence classes of divisors on $\mY$. We define $\on{Div}^+(\mY)_\mathbb{K}$ to be the set of $\mathbb{K}$-divisors on $\mY$ whose weight sequences are summable, and $\on{Div}^b(\mY)_\mathbb{K}$ to be the set of $\mathbb{K}$-divisors on $\mY$ whose weight sequences are bounded.

There is evidently a pairing
$$\on{Div}^b(\mY)_\mathbb{K}\otimes\on{Div}^+(\mY)_\mathbb{K}\to\R,\;\;\;\;\;\mD\cdot\mD'=D_0\cdot D_0'+\sum_{n\geq1}d_nd_n'E_n^2$$
where $(d_n)_{n\in\Z_{\geq1}}$ and $(d_n')_{n\in\Z_{\geq1}}$ are the weight sequences of $\mD$ and $\mD'$ respectively. We choose the codomain to be $\R$ to avoid issues of integrality when $\mY$ is not smooth. This pairing extends the intersection product for each $Y_n$ in the sense that we can view $\on{Div}(Y_n)_\mathbb{K}$ as the subspace of $\on{Div}^+(\mY)_\mathbb{K}\subseteq\on{Div}^b(\mY)_\mathbb{K}$ consisting of all $\mathbb{K}$-divisors on $\mY$ whose weight sequences vanish after the $n$th term. When $\mA$ is a polarisation on a smooth Looijenga tower $\mY$ we will make much use of the quantity
$$-K_\mY\cdot\mA=-K_{Y_0}\cdot A_0-\sum_{a\in\on{wt}(\mY,\mA)}a$$

\subsection{Toric Looijenga towers} \label{sec:tibs}

We will study a class of polarised Looijenga towers arising from weighted posets that come from convex domains in $\R^2_{\geq0}$. The Looijenga pairs constituting these towers are toric surfaces. Key to our construction to is the \emph{weight sequence} $\on{wt}(\Omega)$ associated to a convex domain $\Omega$ following \cite{mcd_sym_09,ccfhr_sym_14}. It is well-known (e.g.~the work of Gross--Hacking--Keel \cite{ghk_mir_15,ghk_mod_15}) that the geometry of Looijenga pairs is close to the geometry of toric varieties and so this is a rich example to consider algebraically, as well as being the main source of applications to symplectic geometry.

We start by recalling the weight sequence associated to a concave or convex domain in $\R^2$. Let $\Delta_a$ denote the triangle in $\R^2$ with vertices $(0,0),(a,0),(0,a)$.

\begin{definition} \label{def:conc_wt} Let $\Omega$ be a concave domain. The \emph{weight sequence} $\on{wt}(\Omega)$ of $\Omega$ is defined recursively as follows.
\begin{itemize}
\item Set $\on{wt}(\emptyset)=\emptyset$ and $\on{wt}(\Delta_a)=(a)$.
\item Otherwise let $a$ be the largest real number such that $\Delta_a\subseteq\Omega$. This divides $\Omega$ into three (possibly empty) pieces: $\Delta_a,\Omega_2',\Omega_3'$.
\item If not empty, $\Omega_2'$ and $\Omega_3'$ are affine-equivalent to concave domains $\Omega_2$ and $\Omega_3$. Define
$$\on{wt}(\Omega)=(a)\cup\on{wt}(\Omega_2)\cup\on{wt}(\Omega_3)$$
regarded as a multiset.
\end{itemize}
\end{definition}

Note that $\on{wt}(\Omega)$ is finite if and only if $\Omega$ is a real multiple of a lattice concave domain but will be infinite in general. We define an analogous sequence for convex domains.

\begin{definition} \label{def:conv_wt} Let $\Omega$ be a convex domain. The \emph{weight sequence} $\on{wt}(\Omega)$ of $\Omega$ is defined recursively as follows.
\begin{itemize}
\item Let $c$ be the smallest real number such that $\Omega\subseteq\Delta_c$.
\item This divides $\Delta_c$ into three (possibly empty) pieces: $\Omega,\Omega_2',\Omega_3'$.
\item If non-empty, $\Omega_2'$ and $\Omega_3'$ are affine-equivalent to concave domains $\Omega_2$ and $\Omega_3$. Define
$$\on{wt}(\Omega)=(c)\cup\on{wt}(\Omega_2)\cup\on{wt}(\Omega_3)$$
using Def.~\ref{def:conc_wt}. We regard this as a multiset with a distinguished element $a$ from the recursion above that we call the \emph{head} of $\on{wt}(\Omega)$. We set $\on{wt}^-(\Omega):=\on{wt}(\Omega)\setminus\{c\}$.
\end{itemize}
\end{definition}

We depict the decompositions used to recursively define the weight sequence in Fig.~\ref{fig:wt_seq}, with the concave case shown in Fig.~\ref{fig:wt_seq}(a) and the convex case in Fig.~\ref{fig:wt_seq}(b). In both cases we denote the parts of $\partial\Delta_a$ away from $\Omega$ by dashed lines.

\begin{figure}[h]
\caption{Weight sequence decompositions}
\label{fig:wt_seq}
\begin{center}
\begin{tikzpicture}[scale=0.6]
\foreach \i in {0,...,4}
{
\foreach \j in {0,...,5}
{\node (\i\j) at (\i,\j){\tiny $\bullet$};
\node (\i\j) at (8+\i,\j){\tiny $\bullet$};
}
\node (\i) at (13,\i){\tiny $\bullet$};
}

\node (5) at (13,5){\tiny $\bullet$};

\draw (0,0) to (4,0);
\draw (0,0) to (0,5) to (1,2) to (2,1) to (4,0);
\draw[dashed] (0,3) to (3,0);

\node (l2) at (5,1){\small $\Omega_3'$};
\node (l1) at (5,4.7){\small $\Omega_2'$};

\draw[->] (l2) to (3.5,0.5);
\draw[->] (l1) to (0.7,3.5);

\draw (8,0) to (12,0);
\draw (8,0) to (8,4) to (10,3) to (12,1) to (12,0);
\draw[dashed] (8,5) to (13,0);
\draw[dashed] (8,4) to (8,5);
\draw[dashed] (12,0) to (13,0);

\node (l2) at (14,1){\small $\Omega_3'$};
\node (l1) at (14,4.2){\small $\Omega_2'$};

\draw[->] (l2) to (12.5,0.7);
\draw[->] (l1) to (8.7,4.5);

\node (la) at (2,-0.8){(a)};
\node (lb) at (10.5,-0.8){(b)};
\end{tikzpicture}
\end{center}
\end{figure}

\begin{definition} Let $(\mY,\mA)$ be a pseudo-polarised Looijenga tower. We say $(\mY,\mA)$ is \emph{toric} if $(Y_0,A_0)$ is a toric surface polarised by a torus-invariant $\R$-divisor and each blowup map $\pi_n$ is equivariant.
\end{definition}

We associate a toric pseudo-polarised Looijenga tower $(\mY_\Omega,\mA_\Omega)$ to a convex domain $\Omega$. This will have the property $\on{wt}(\Omega)=\on{wt}(\mY_\Omega,\mA_\Omega)$. We write $\on{wt}(\Omega)=\{c\}\cup\on{wt}^-(\Omega)$.

Consider $\pr^2$ with moment image $\Delta$ shown in Fig.~\ref{fig:p2} where the lower left vertex is the origin. We denote the hyperplanes corresponding to the three edges of $\Delta$ by $H_0,H_1,H_2$ as shown.

\begin{figure}[h]
\caption{Moment polytope of $\pr^2$}
\label{fig:p2}
\begin{center}
\begin{tikzpicture}[scale=1]
\foreach \i in {0,...,1}
{
\foreach \j in {0}
{\node (\i\j) at (\i,\j){\tiny $\bullet$};
}
\node (0\i) at (0,\i){\tiny $\bullet$};
}

\draw (0,0) to (0,1) to (1,0) to (0,0);

\node (l0) at (1.4,1.4){\small $H_0$};
\node (l1) at (-1,0.5){\small $H_2$};
\node (l2) at (0.5,-1){\small $H_1$};

\draw[->] (l0) to (0.45,0.65);
\draw[->] (l2) to (0.5,-0.05);
\draw[->] (l1) to (-0.05,0.5);
\end{tikzpicture}
\end{center}
\end{figure}
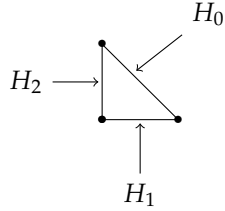

We start with $(Y,L,A)=(\pr^2,H_0+H_1+H_2,cH_0)$. Set $\mP_\Omega=\mP_{(Y,L)}$. We will construct a big and nef weight function on $\mP_\Omega$ via the recursion defining the weight sequence for $\Omega$, and hence a (toric) pseudo-polarised Looijenga tower.

Each element $p\in\mP_\Omega$ by definition corresponds to a node on a toric blowup of $(Y,L)$ but from Def.~\ref{def:conc_wt} and Def.~\ref{def:conv_wt} $p$ also corresponds to a step in the weight sequence recursion. Recall the construction of $\mP_{(Y,L)}=\bigcup_{n\geq0}\mP_n$. In this notation the elements of $\mP_0$ correspond to the three torus-fixed points of $\pr^2$. We assign weight zero to the torus-fixed point $p_1=H_1\cap H_2$ whose moment image is the origin and to all its descendants, capturing the fact that there will be no blowups performed with that centre.

The two other points $p_2,p_3\in\mP_0$ correspond to the concave domains $\Omega_2$ and $\Omega_3$ from Def.~\ref{def:conv_wt}. Set $\on{wt}(p_i)=a_i$, where $\Delta_{a_i}$ is the largest regular triangle that fits inside $\Omega_i$ for $i=2,3$. Iterating this procedure assigns a weight to each element of $\mP_\Omega$ as the side length of the largest regular triangle that fits inside the corresponding concave domain.

More precisely, we fix notation as follows. Let $\Omega_2$ and $\Omega_3$ be as above. Applying the weight sequence recursion to $\Omega_2$ yields two concave domains $\Omega_{22}$ and $\Omega_{23}$ and similarly applying it to $\Omega_3$ yields concave domains $\Omega_{32}$ and $\Omega_{33}$. Repeating this process yields the diagram in Fig.~\ref{fig:wt_rec}(a). Notice that this is naturally in bijection with the part of the Haase diagram of the poset $\mP_\Omega$ excluding the $2$-valent tree with maximum $p_1$. We denote by $\Delta(q)$ the concave domain (i.e.~either $\Omega_2$ or $\Omega_3$ in Def.~\ref{def:conc_wt}) corresponding to $q\in\mP_\Omega\setminus\{q\in\mP_\Omega:q<p_1\}$.

We define a weight function on $\mP_\Omega$ by
$$\on{wt}(q)=\begin{cases}
0 & q<p_1 \\
a(q) & \text{else}
\end{cases}$$
where $\Delta_{a(q)}$ is the largest regular triangle contained in $\Omega_q$. This is shown in Fig.~\ref{fig:wt_rec}(b) with the same indexing as in Fig.~\ref{fig:wt_rec}(a). Observe that this weight function is big and nef since the associated polarised toric surface $(Y_n,A_n)$ corresponds to the polytope $\Omega_n$ obtained after the $n$th step of the weight sequence recursion; for comparison, see \cite[\S3.2-3.3]{cri_sym_19}.

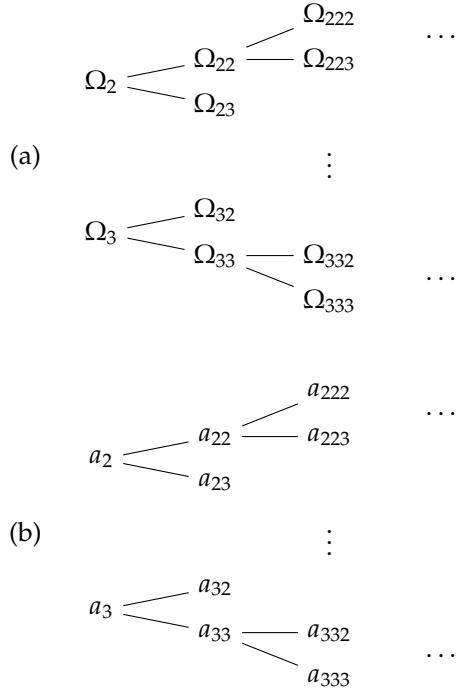
\begin{figure}[h]
\caption{Weight sequence recursion}
\label{fig:wt_rec}
\begin{center}
\begin{tikzpicture}
\small
\node (la) at (-1,0){(a)};
\node (la) at (-1,-5){(b)};

\node (a1) at (0,1){$\Omega_2$};
\node (a2) at (0,-1){$\Omega_3$};

\node (b1) at (1.5,1.3){$\Omega_{22}$};
\node (b2) at (1.5,0.7){$\Omega_{23}$};
\node (b3) at (1.5,-0.7){$\Omega_{32}$};
\node (b4) at (1.5,-1.3){$\Omega_{33}$};

\node (c1) at (3,1.9){$\Omega_{222}$};
\node (c2) at (3,1.3){$\Omega_{223}$};

\node (e1) at (3,0){$\vdots$};

\node (c7) at (3,-1.3){$\Omega_{332}$};
\node (c8) at (3,-1.9){$\Omega_{333}$};

\node (e2) at (4.5,1.6){$\dots$};
\node (e3) at (4.5,-1.6){$\dots$};

\draw (b1) to (a1);
\draw (b2) to (a1);
\draw (b3) to (a2);
\draw (b4) to (a2);
\draw (c1) to (b1);
\draw (c2) to (b1);
\draw (c7) to (b4);
\draw (c8) to (b4);

\node (a1) at (0,-4){$a_2$};
\node (a2) at (0,-6){$a_3$};

\node (b1) at (1.5,1.3-5){$a_{22}$};
\node (b2) at (1.5,0.7-5){$a_{23}$};
\node (b3) at (1.5,-0.7-5){$a_{32}$};
\node (b4) at (1.5,-1.3-5){$a_{33}$};

\node (c1) at (3,1.9-5){$a_{222}$};
\node (c2) at (3,1.3-5){$a_{223}$};

\node (e1) at (3,0-5){$\vdots$};

\node (c7) at (3,-1.3-5){$a_{332}$};
\node (c8) at (3,-1.9-5){$a_{333}$};

\node (e2) at (4.5,1.6-5){$\dots$};
\node (e3) at (4.5,-1.6-5){$\dots$};

\draw (b1) to (a1);
\draw (b2) to (a1);
\draw (b3) to (a2);
\draw (b4) to (a2);
\draw (c1) to (b1);
\draw (c2) to (b1);
\draw (c7) to (b4);
\draw (c8) to (b4);
\end{tikzpicture}
\end{center}
\end{figure}

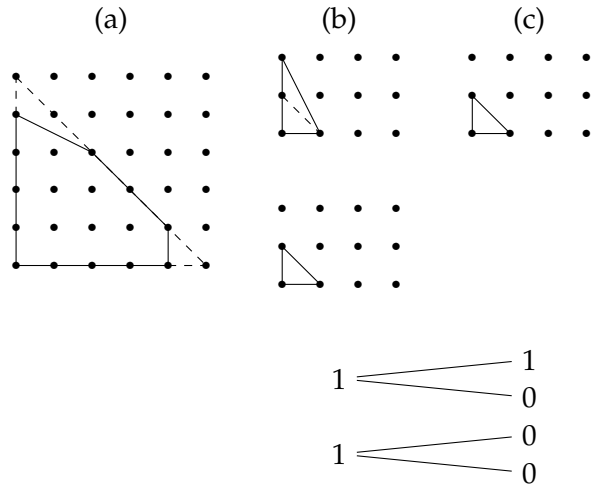
\begin{figure}[h]
\caption{Constructing $\mP_\Omega$}
\label{fig:wt_poset}
\begin{center}
\begin{tikzpicture}[scale=0.5]
\foreach \i in {0,...,5}
{
\foreach \j in {0,...,5}
{
\node (\i\j) at (\i,\j){\tiny $\bullet$};
}
}

\draw (0,0) to (4,0);
\draw (0,0) to (0,4) to (2,3) to (4,1) to (4,0);
\draw[dashed] (0,5) to (5,0);
\draw[dashed] (0,4) to (0,5);
\draw[dashed] (4,0) to (5,0);

\foreach \k in {0,...,1}
{
\foreach \i in {0,...,3}
{
\foreach \j in {0,...,2}
{
\node (a) at (7+\i,4*\k+\j-1/2){\tiny $\bullet$};
}
}
}

\draw (7,5.5) to (7,3.5) to (8,3.5) to (7,5.5);
\draw[dashed] (7,4.5) to (8,3.5);
\draw (7,-0.5) to (7,0.5) to (8,-0.5) to (7,-0.5);

\node (p1) at (8.5,-3){$1$};
\node (p2) at (8.5,-5){$1$};

\foreach \k in {1}
{
\foreach \i in {0,...,3}
{
\foreach \j in {0,...,2}
{
\node (a) at (12+\i,4*\k+\j-1/2){\tiny $\bullet$};
}
}
}

\draw (12,4.5) to (12,3.5) to (13,3.5) to (12,4.5);

\node (p11) at (13.5,-2.5){$1$};
\node (p12) at (13.5,-3.5){$0$};

\node (p21) at (13.5,-4.5){$0$};
\node (p22) at (13.5,-5.5){$0$};

\draw (p11) to (p1);
\draw (p12) to (p1);
\draw (p21) to (p2);
\draw (p22) to (p2);

\node (la) at (2.5,6.5){(a)};
\node (lb) at (8.5,6.5){(b)};
\node (lc) at (13.5,6.5){(c)};
\end{tikzpicture}
\end{center}
\end{figure}

\begin{example}
We will work out the construction of $\mP_\Omega$ in detail for the convex domain $\Omega$ from Fig.~\ref{fig:wt_seq}(b) with weight sequence $(5;1,1,1)$. In Fig.~\ref{fig:wt_poset}(a) we show $\Omega$ with the first step of the weight sequence recursion expressing the $4$-ball $B^4(5)$ as a union of $X_\Omega,B^4(1)$ and the ellipsoid $E(1,2)$. In Fig.~\ref{fig:wt_poset}(b) we show the two regions $\Omega_2$ and $\Omega_3$. The weights for elements of $\mP_0\subseteq\mP_\Omega$ are illustrated below the domains. In Fig.~\ref{fig:wt_poset}(c) the final stage of the recursion is shown -- of the four concave domains coming from $\Omega_2$ and $\Omega_3$ only one is nonempty, and is equal to $\Delta_1$ -- and the corresponding weights are listed below. Throughout we omit the tree with maximum $p_1$ with weights all zero.
\end{example}

This example terminated after finitely many stages because $\Omega$ was a rational-sloped polytope. In this case all $(Y_n,A_n)$ for large enough $n$ are equivalent to the polarised toric surface $(Y_\Omega,A_\Omega)$ associated to $\Omega$. That is, for large enough $n$ there is a series of toric blowups $\pi\colon Y_n\to Y_\Omega$ with $A_n=\pi^*A_\Omega$.

Note that the sequence of polarised toric surfaces $(Y_n,A_n)$ produced using the structure of $\mP_{(Y,L)}$ as a graded poset from \S\ref{sec:wt_poset} recovers the sequence of approximations used by Hutchings in \cite[\S3]{hut_ech_19} by setting $\Omega_n$ to be the polytope for $A_n$.

\subsection{Algebraic capacities for Looijenga towers} \label{sec:alg_cap}

Recall the construction of \emph{algebraic capacities} for a $\Q$-factorial pseudo-polarised surface $(Y,A)$:
$$\calg_k(Y,A):=\inf_{\on{Nef}(Y)_\Z}\{D\cdot A:\chi(D)\geq k+\chi(\mO_Y)\}$$
When $Y$ is smooth this reduces to
$$\calg_k(Y,A):=\inf_{\on{Nef}(Y)_\Z}\{D\cdot A:I(D)\geq2k\}$$
where $I(D):=D\cdot(D-K_Y)$. It was shown in \cite[Prop.~2.11]{wor_alg_20} for all smooth or toric pseudo-polarised surfaces $(Y,A)$ we have that $\calg_k(Y,A)$ is obtained by ranging over effective $\Z$-divisors in place of nef $\Z$-divisors.

\begin{lemma} \label{lem:lim_alg_cap} Suppose $(\mY,\mA)=\{(Y_n,L_n,A_n)\}_{\Z_{\geq0}}$ is a pseudo-polarised Looijenga tower that is smooth or toric. Then
$$\lim_{n\to\infty}\calg_k(Y_n,A_n)$$
exists and is finite.
\end{lemma}

This result also holds for any tower $(\mY,\mA)=\{(Y_n,A_n)\}$ of pseudo-polarised surfaces related by blowups. Notice that in this case $\chi(\mO_{Y_n})=\chi(\mO_{Y_m})=:\chi(\mO_\mY)$ for all $n,m\in\Z_{\geq0}$.

\begin{proof} Let $D_n$ be a nef $\Z$-divisor computing $\calg_k(Y_n,A_n)$. We have in the smooth or toric cases that $\chi(\pi_{n+1}^*D_n)=\chi(D_n)\geq k+\chi(\mO_\mY)$ so that
$$\calg_k(Y_{n+1},A_{n+1})\leq\pi_{n+1}^*D_n\cdot A_{n+1}=D_n\cdot A_n=\calg_k(Y_n,A_n)$$
It follows that $\calg_k(Y_n,A_n)$ is a decreasing sequence in $n$ that is bounded below, and is hence convergent.
\end{proof}

We thus define
$$\calg_k(\mY,\mA):=\lim_{n\to\infty}\calg_k(Y_n,A_n)$$
We will see in the next section that this definition extends the relationship between algebraic capacities of polarised algebraic surfaces and ECH capacities of related symplectic $4$-manifolds. Next we note how our notion of equivalence from \S\ref{sec:loo} was motivated by the structure of algebraic capacities.

\begin{lemma} \label{lem:equiv} Let $(Y,L,A)$ and $(Y',L',A')$ be smooth or toric Looijenga pairs. If $(Y,L,A)$ and $(Y',L',A')$ are equivalent, then
$$\calg_k(Y,A)=\calg_k(Y',A')$$
for all $k\in\Z_{\geq0}$.
\end{lemma}

This is essentially the content of \cite[Prop.~3.4 + Prop.~3.5]{wor_alg_20}. It is clear that the anticanonical divisors play no role in this result. The value of Lem.~\ref{lem:equiv} is in allowing us to fix a particular universal Looijenga tower and choose a weight function on it to calculate the algebraic capacities of any pseudo-polarised Looijenga tower. We will hence also not specify the function $h$ we have chosen to produce a bona fide Looijenga tower (indexed by $\Z_{\geq0}$) from the poset $\mP_{(Y,L)}$.

We end this subsection by showing that one can capture the algebraic capacities of $(\mY,\mA)$ intrinsically in terms of divisors on $\mY$. Define $\on{Nef}(\mY)$ to be the submonoid of $\on{Div}^+(\mY)_\R$ consisting of divisors $\mD$ such that $\mD\cdot E\geq0$ for all $E\in\on{\ol{NE}}(Y_n)$ for each $n$. Note that $\on{Nef}(Y_n)$ naturally embeds into $\on{Nef}(\mY)$. Set $\on{Nef}(\mY)_\Z=\on{Nef}(\mY)\cap\on{Div}(\mY)_\Z$.

\begin{prop} \label{prop:intrinsic} If $(\mY,\mA)$ is a smooth or toric pseudo-polarised Looijenga tower, then
$$\calg_k(\mY,\mA)=\inf_{\mD\in\on{Nef}(\mY)_\Z}\{\mD\cdot\mA:\mD\cdot(\mD-K_\mY)\geq2k\}$$
\end{prop}

We write $\mD=(D_0,d_1,\dots)$ for a $\mathbb{K}$-divisor on $\mY$ where $D_0$ is a $\mathbb{K}$-divisor on $Y_0$ and $\{d_n\}_{n\in\Z_{\geq1}}$ is the weight sequence of $\mD$. The assumption that $\mY$ is smooth or toric allows us by Lem.~\ref{lem:equiv} to reduce to the smooth case where the constraint is given in terms of $I(D)$.

\begin{proof} Since $\on{Nef}(Y_n)_\Z$ can be viewed as a subset of $\on{Nef}(\mY)_\Z$ we obtain
$$\calg_k(Y_n,A_n)\geq\inf_{\mD\in\on{Nef}(\mY)_\Z}\{\mD\cdot\mA:\mD\cdot(\mD-K_\mY)\geq2k\}$$
for each $n$, and so
$$\calg_k(\mY,\mA)\geq\inf_{\mD\in\on{Nef}(\mY)_\Z}\{\mD\cdot\mA:\mD\cdot(\mD-K_\mY)\geq2k\}$$
For the converse it suffices that for each $\eps>0$ we can find $N\in\Z_{\geq0}$ such that
$$\calg_k(Y_n,A_n)\leq\inf_{\mD\in\on{Nef}(\mY)_\Z}\{\mD\cdot\mA:\mD\cdot(\mD-K_\mY)\geq2k\}+\eps$$
for all $n>N$. Let $\mD_0=(D_0,d_1,\dots)\in\on{Nef}(\mY)_\Z$ be such that
$$\mD_0\cdot\mA\leq\inf_{\mD\in\on{Nef}(\mY)_\Z}\{\mD\cdot\mA:\mD\cdot(\mD-K_\mY)\geq2k\}+\eps$$
As $\mD_0\in\on{Div}^+(\mY)$ we must have that $d_i=0$ for all $i>N$ for some $N\in\Z_{\geq0}$. There is thus a nef $\Z$-divisor $D_n\in\on{Nef}(Y_n)$ for $n>N$ that is mapped to $\mD_0$ under the embedding $\on{Nef}(Y_n)\to\on{Nef}(\mY)$. Hence, for all $n>N$,
$$\calg_k(Y_n,A_n)\leq D_n\cdot A_n=D_n\cdot\mA=\mD_0\cdot\mA\leq\inf_{\mD\in\on{Nef}(\mY)_\Z}\{\mD\cdot\mA:\mD\cdot(\mD-K_\mY)\geq2k\}+\eps$$
as required.
\end{proof}

In fact this infimum is realised in the toric case via a sympletic argument using Prop.~\ref{prop:ech_alg} below.

\section{Sub-leading asymptotics of ECH capacities}

\subsection{Looijenga towers and ECH}

To each symplectic $4$-manifold $(X,\omega)$ ECH associates an increasing sequence
$$\{\cech_k(X,\omega)\}_{k\in\Z_{\geq0}}$$
of (extended) real numbers called the \emph{ECH capacities} of $(X,\omega)$. These obstruct symplectic embeddings in the sense that
$$(X,\omega)\se(X',\omega')\Longrightarrow\cech_k(X,\omega)\leq\cech_k(X',\omega')\text{ for all $k$}$$

\begin{prop} \label{prop:ech_alg} For any convex domain $\Omega\subseteq\R^2$ the toric polarised Looijenga tower $(\mY_\Omega,\mA_\Omega)$ has
$$\cech_k(X_\Omega)=\calg_k(\mY_\Omega,\mA_\Omega)$$
\end{prop}

\begin{proof} Consider the sequence of polygons $\{\Omega_n\}_{n\in\Z_{\geq0}}$ arising as the polytopes associated to the divisors $A_n$. We know that $\lim_{n\to\infty}\cech_k(X_{\Omega_n})=\cech_k(X_\Omega)$ by Hausdorff continuity \cite[Lem.~2.3]{ccfhr_sym_14}. Since $\Omega_n$ is rational-sloped \cite[Thm.~1.5]{wor_ech_19} gives that $\cech_k(X_{\Omega_n})=\calg_k(Y_n,A_n)$ and so the result follows from Lem.~\ref{lem:lim_alg_cap}.
\end{proof}

Using the same sequence of approximations we prove a result similar to \cite[Lem.~3.6]{hut_ech_19}.

\begin{prop}[{c.f.~\cite[Lem.~3.6]{hut_ech_19}}] Let $\Omega$ be a convex domain whose weight sequence has head $c$. Then
$$3c-\sum_{a\in\on{wt}(\Omega)}a=a(\Omega)+b(\Omega)+\laff(\partial^+\Omega)$$
\end{prop}

\begin{proof} Note that
$$3c-\sum_{a\in\on{wt}(\Omega)}a=-K_{\mY_\Omega}\cdot\mA_\Omega=\lim_{n\to\infty}-K_{Y_n}\cdot A_n=\lim_{n\to\infty}a(\Omega_n)+b(\Omega_n)+\laff(\partial^+\Omega_n)$$
and the result follows from continuity of $a(\cdot)$ and $b(\cdot)$ and analysis similar to \cite[Lem.~3.6]{hut_ech_19}.
\end{proof}

\subsection{Asymptotics for algebraic capacities}

Just like for ECH capacities of symplectic $4$-manifolds and algebraic capacities of pseudo-polarised algebraic surfaces we have a `Weyl law' controlling the growth of $\calg_k(\mY,\mA)$.

\begin{thm} \label{thm:ind_weyl} Let $(\mY\mA)$ be a pseudo-polarised Looijenga tower. Then
$$\lim_{k\to\infty}\frac{\calg_k(\mY,\mA)^2}{k}=2\mA^2$$
\end{thm}

We will not prove this directly, but will instead appeal to the analysis of the \emph{error terms}
$$\ealg_k(\mY,\mA):=\calg_k(\mY,\mA)-\sqrt{2\mA^2k}$$
below, where we will show that $\ealg_k(\mY,\mA)=O(1)$. Thm.~\ref{thm:ind_weyl} follows immediately from Prop.~\ref{prop:ech_alg} when $(\mY,\mA)$ is a toric polarised Looijenga tower arising from a convex domain.

These error terms associated to $(\mY,\mA)$ are analogous to the error terms in ECH
$$e_k(X,\omega):=\cech_k(X,\omega)-\sqrt{4\on{vol}(X,\omega)k}$$
and agree when $(\mY,\mA)$ comes from a convex domain.

\subsection{Bounds for error terms}

For a pseudo-polarised Looijenga tower $(\mY,\mA)$ we define a divisor $-K_\mY^+$ on $\mY$ by
$$-K_\mY^+=-K_\mY+K_{Y_0}-K_{Y_0}^+$$
where $-K_{Y_0}^+$ is the support of $A_0$ viewed as a reduced divisor. As a sequence of divisors indexed by $n$ like in \S\ref{sec:div}, $-K_\mY^+$ has as its $n$th term the support of $A_n$ viewed as a reduced divisor. In this sense $-K_\mY^+$ can be viewed as the `support' of $\mA$. When $(\mY,\mA)$ is actually a pseudo-polarised toric surface corresponding to a rational-sloped polygon $\Omega$, we see that $-K_\mY^+$ is the preimage of $\partial^+\Omega$ under the moment map, giving $-K_\mY^+\cdot\mA=\laff(\partial^+\Omega)$. Our next aim is to prove the following theorem.

\begin{thm} \label{thm:loo_main}
Suppose $(\mY,\mA)$ is a pseudo-polarised Looijenga tower such that $\mY$ is smooth or toric. Then
\begin{align*}
\frac{1}{2}K_\mY\cdot\mA-K_\mY^+\cdot\mA&\geq\limsup_{k\to\infty}\ealg_k(\mY,\mA) \\
&\geq\liminf_{k\to\infty}\ealg_k(\mY,\mA)\geq\frac{1}{2}K_\mY\cdot\mA
\end{align*}
In particular, $\ealg_k(\mY,\mA)=O(1)$.
\end{thm}

\begin{cor} \label{cor:main} Let $X_\Omega$ be a convex toric domain. Then
\begin{align*}
-\frac{1}{2}\left(a(\Omega)+b(\Omega)-\frac{1}{2}\laff(\partial^+\Omega)\right)&\geq\limsup_{k\to\infty}e_k(\mY,\mA) \\
&\geq\liminf_{k\to\infty}e_k(\mY,\mA)\geq-\frac{1}{2}\left(a(\Omega)+b(\Omega)+\frac{1}{2}\laff(\partial^+\Omega)\right)
\end{align*}
When $\Omega$ has no rational-sloped edges we have that $e_k(X_\Omega)$ is convergent and
$$\lim_{k\to\infty}e_k(X_\Omega)=-\frac{1}{2}(a(\Omega)+b(\Omega))$$
\end{cor}

Over the next two subsections we will establish these asymptotic upper and lower bounds. Since $\laff(\partial^+\Omega)=0$ when $\partial^+\Omega$ has no rational-sloped edges the criterion for convergence follows immediately.

We will assume that $\mY$ is smooth, passing to the singular toric case by \cite[Prop.~4.19]{wor_alg_20} that easily extends to the case of toric Looijenga towers.

\subsection{Upper bound for error terms}

Observe that any nef $\Z$-divisor $D$ on a $\Q$-factorial surface $Y$ gives an upper bound
$$\calg_k(Y,A)\leq D\cdot A$$
when $2k\leq I(D)$. By \cite[Prop.~2.11]{wor_alg_20} this also works if $D$ is an effective $\Z$-divisor. Let $(\mY,\mA)=\{(Y_n,L_n,A_n)\}_{n\in\Z_{\geq0}}$ be a pseudo-polarised Looijenga tower. We obtain an upper bound for $\calg_k(Y_n,A_n)$ in terms of $k$ and $n$ by using $\Z$-divisors of the form $\lc dA_n\rc$ and then considering how the resulting bound behaves as $n$ and $k$ become large. We let the components of $A_n$ be denoted $D_1,\dots,D_s$; that is, $-K_{Y_n}^+=\sum_{i=1}^s D_i$.

Consider the constraint
$$I(\lc dA_n\rc)=(dA_n+\Delta_n)\cdot(dA_n+\Delta_n-K_{Y_n})\geq 2k$$
where $\Delta_n=\lc dA_n\rc-dA_n$. That is,
$$d^2A_n^2-dA_n\cdot K_{Y_n}+2dA_n\cdot\Delta_n-2k+\Delta_n^2-\Delta_n\cdot K_{Y_n}\geq0$$
Notice that $2dA_n\cdot\Delta_n\geq0$ since $\Delta_n$ is effective and so we ignore that term. We bound $\Delta_n\cdot\Delta_n-\Delta_n\cdot K_{Y_n}$ in terms of the geometry of $Y_n$. Notice that
$$\Delta_n^2\geq\sum_{D_i^2<0}D_i^2$$
and
$$-\Delta_n\cdot K_{Y_n}\geq\sum_{D_i^2<-1}(2+D_i^2)$$
giving
$$\Delta_n^2-\Delta\cdot K_{Y_n}\geq-\#\{\text{$(-1)$-curves on $Y_n$}\}+2\sum_{D_i^2<-1}(1+D_i^2)$$
Hence we see that $I(\lc dA_n\rc)\geq 2k$ when
$$d^2A_n^2-dA_n\cdot K_{Y_n}-2k-\#\{\text{$(-1)$-curves on $Y_n$}\}+2\sum_{D_i^2<-1}(1+D_i^2)\geq0$$
or when $d$ is bounded below by the larger solution of the quadratic obtained by replacing $\geq$ with $=$ in the above. Write $-A_n\cdot K_{Y_n}/A_n^2=:\kappa_n$. We thus have $I(\lc dA_n\rc)\geq 2k$ if
$$d\geq-\frac{\kappa_n}{2}+\sqrt{\frac{2k}{A_n^2}+\frac{\#\{\text{$(-1)$-curves on $Y_n$}\}-2\sum_{D_i^2<-1}(1+D_i^2)}{A_n^2}+\frac{\kappa_n^2}{4A_n^2}}$$
Set
$$F(n)=\#\{\text{$(-1)$-curves on $Y_n$}\}-2\sum_{D_i^2<-1}(1+D_i^2)$$
We study how $F(n)$ changes with $n$ by measuring $F(n+1)-F(n)$; i.e.~how $F$ changes under a single blowup in a torus-fixed point between two torus-invariant curves $C_1$ and $C_2$. Let $\{i,j\}=\{1,2\}$. The options are:
\begin{itemize}
\item $C_1^2>0$ and $C_2^2>0\Longrightarrow F(n+1)-F(n)=1$.
\item $C_i^2=0$ and $C_j^2>0\Longrightarrow F(n+1)-F(n)=2$.
\item $C_i^2=-1$ and $C_j^2>0\Longrightarrow F(n+1)-F(n)=2$.
\item $C_i^2\leq-2$ and $C_j^2>0\Longrightarrow F(n+1)-F(n)=3$.
\item $C_i^2=-1$ and $C_j^2=0\Longrightarrow F(n+1)-F(n)=3$.
\item $C_i^2==1$ and $C_j^2=-1\Longrightarrow F(n+1)-F(n)=3$.
\item $C_i^2\leq-2$ and $C_j^2=0\Longrightarrow F(n+1)-F(n)=4$.
\item $C_i^2\leq-2$ and $C_j^2=-1\Longrightarrow F(n+1)-F(n)=4$.
\item $C_i^2\leq-2$ and $C_j^2\leq-2\Longrightarrow F(n+1)-F(n)=5$.
\end{itemize}
Hence we see that $F(n+1)-F(n)\leq 5$. In the toric case we have $Y_0=\pr^2$ and so $F(n)\leq 5n$ since $\pr^2$ has no negative curves. In general we will have $F(n)\leq 5n+F(0)$ but, as it makes no significant difference to the argument, we will ignore the constant for notational convenience.

Therefore we see that $I(\lc dA_n\rc)\geq 2k$ when
$$d\geq-\frac{\kappa_n}{2}+\sqrt{\frac{2k}{A_n^2}+\frac{5n}{A_n^2}+\frac{\kappa_n^2}{4A_n^2}}=:d_{k,n}$$
It follows that
\begin{align*}
\calg_k(Y_n,A_n)\leq\lc d_{k,n}A_n\rc\cdot A_n&\leq d_{k,n}A_n^2-K_{Y_n}^+\cdot A_n \\
&=-\frac{\kappa_nA_n^2}{2}+\sqrt{2A_n^2k+5A_n^2n+\frac{\kappa_n^2(A_n^2)^2}{4}}-K_{Y_n}^+\cdot A_n
\end{align*}
This is an explicit bound for $\calg_k(Y_n,A_n)$ valid for all $k$ and $n$. We require an elementary lemma from analysis to study what happens as $n$ and $k$ get large.

\begin{lemma} \label{lem:analysis} Suppose $(a_i)$ is a decreasing summable sequence. Let $S(n)=\sum_{i\geq n}a_i^2$. Then there exists a strictly increasing sequence $(n_k)$ of natural numbers such that $n_k=o(\sqrt{k})$ and $S(n_k)=o(1/\sqrt{k})$.
\end{lemma}

We use some basic techniques from probability theory to prove this result, though a rather longer but completely elementary proof also exists. Notice that it makes no difference to demand that $n_k=o(k)$ and $S(n_k)=o(1/k)$ instead of $n_k=o(\sqrt{k})$ and $S(n_k)=o(1/\sqrt{k})$, which we adopt for notational convenience.

\begin{proof} We first show that $a_i=o(i)$. We can choose $(a_i)$ to be non-increasing and so we may interpret it as the tail probabilities $a_i= P(X > i)$ for some random variable $X$ with values in $\N$. As $a_i$ is summable, $X$ has finite expectation:
$\mathbb{E}X = \sum_{i}P(X > i) = \sum_i a_i < \infty$. Now,
$$k\cdot a_k = k\cdot P(X > k) = k\cdot \mathbb E 1_{X > k} = \mathbb E k1_{X>k} \leq \mathbb E X1_{X>k}$$
which approaches $0$ as $k\to\infty$ by the dominated convergence theorem. It follows that $a_i \leq b_i/i$ for some $b_i \in o(1)$, which again without loss of generality we may choose to be decreasing. We now define $T_k = \inf\{t : \sum_{i\geq t} a^2_i \leq k^{-1}\}$. 

\begin{claim} $T_k=o(k)$.
\end{claim}

Computing tails we find, using the monotonicity of $(b_i)$,
\begin{equation} \tag{$\spadesuit$}
\sum_{i\geq t} a^2_i \leq \sum_{i\geq t} b^2_i/i^2 \leq b^2_t \sum_{i\geq t} i^{-2} = b^2_t/t,
\label{eq:tail_sum}
\end{equation}
Define $S_k = \inf\{ s : b_s^2/s \leq k^{-1} \}$. We see that $T_k \leq S_k$, so it suffices that $S_k=o(k)$. But by definition, $(S_k - 1) / k < b^2_{S_k - 1} \in o(1)$ and so we have shown the claim.

To finish the proof, we know from \eqref{eq:tail_sum} that $S(t)=o(t^{-1})$, i.e. $S(t) \leq f_t/t$ for some non-increasing $f_t \in o(1)$. Consequently, we are looking for a sequence $n_k=o(k)$ such that $f_{n_k} / n_k=o\left(k^{-1}\right)$, or equivalently for a sequence $m_k=o(1)$ for which $f_{km_k} / m_k = o(1)$. Here is a construction of such a sequence $m_k$:
\begin{itemize}
\item Define $t_j = \inf\{ t : f_t \leq 2^{-j}\}$.
\item Set $m_k = \sum_{j = 0}^{\infty}1_{\left\{jt_j \leq m_k < (j+1)t_{j+1}\right\}} j^{-1}$.
\end{itemize}
These $m_k$ are certainly $o(1)$, and with $j(k) = \sup\{ j : j t_j \leq k\}$ we have
$$f_{km_k} / m_k = f_{k\cdot j(k)^{-1}} \cdot j(k) \leq f_{t_{j(k)}} \cdot j(k) = j(k)\cdot 2^{-j(k)}=o(1)$$
as desired.
\end{proof}

In this context Lemma \ref{lem:analysis} implies that there is a function $n(k)$ that that depends only on $(\mY,\mA)$ and is $o(\sqrt{k})$ such that $\sum_{i\geq n(k)}a_i^2=o(1/\sqrt{k})$. It follows that $|A_{n_k}^2-\mA|=o(1/\sqrt{k})$. Since $\calg_k(\mY,\mA)\leq\calg_k(Y_n,A_n)$ for all $n$ and $k$ we get
{\small
\begin{align*}
&\ealg_k(\mY,\mA)\leq-\frac{\kappa_{n(k)}A_{n(k)}^2}{2}-K_{Y_{n(k)}}^+\cdot A_{n(k)}+\sqrt{2A_{n(k)}^2k+5A_{n(k)}^2n(k)+\frac{(\kappa_{n(k)}A_{n(k)})^2}{4}}-\sqrt{2\mA^2k} \\
&=-\frac{\kappa_{n(k)}A_{n(k)}^2}{2}-K_{Y_{n(k)}}^+\cdot A_{n(k)}+\sqrt{2\left(\mA^2+o\left(\frac{1}{\sqrt{k}}\right)\right)k+5\left(\mA^2+o\left(\frac{1}{\sqrt{k}}\right)\right)(n(k)+1)+O(1)}-\sqrt{2\mA^2k} \\
&=-\frac{\kappa_{n(k)}A_{n(k)}^2}{2}-K_{Y_{n(k)}}^+\cdot A_{n(k)}+\sqrt{2\mA^2k+o(\sqrt{k})}-\sqrt{2\mA^2k}
\end{align*}}

By letting $k\to\infty$ and substituting $\kappa_nA_n^2=-A_n\cdot K_{Y_n}$ we achieve the following.

\begin{prop} Let $(\mY,\mA)$ be a pseudo-polarised Looijenga tower. Then,
$$\limsup_{k\to\infty}\ealg_k(\mY,\mA)\leq\frac{1}{2}K_\mY\cdot\mA-K_\mY^+\cdot\mA$$
\end{prop}

We convert this into combinatorial language.

\begin{cor} Let $\Omega$ be a convex domain. Then,
$$\limsup_{k\to\infty}e_k(X_\Omega)\leq-\frac{1}{2}\left(a(\Omega)+b(\Omega)-\frac{1}{2}\laff(\partial^+\Omega)\right)$$
In particular, if $\partial^+\Omega$ has no rational-sloped edge then
$$\limsup_{k\to\infty}e_k(X_\Omega)\leq-\frac{1}{2}(a(\Omega)+b(\Omega))$$
\end{cor}

\subsection{Lower bound for error terms}

To deduce a lower bound we can in fact generalise to the setting of a tower of blowups $\mY=\{(Y_n,A_n)\}_{n\in\Z_{\geq0}}$ of polarised surfaces where $-K_\mY$ is `effective' -- that is, each $-K_{Y_n}$ is effective. Denote
$$\on{NS}(Y)_{A\geq0}:=\{D\in\on{NS}(Y):D\cdot A\geq0\}$$
Define for a pseudo-polarised surface $(Y,A)$
$$c_k^+(Y,A):=\inf_{\on{NS}(Y)_{A\geq0}}\{D\cdot A:D\cdot(D-K_Y)\geq 2k\}$$
This is a variation on the asymptotic capacity $\casy_k(Y,A)$ from \cite[\S4.1]{wor_alg_20} or the estimate using the `approximate ECH index' of \cite[\S5.2]{hut_ech_19}. These invariants will have preferable numerics to study lower bounds for $\calg_k(Y,A)$. It is already clear that
$$c_k^+(Y,A)\leq\calg_k(Y,A)$$
for all $k$. As usual we write $\kappa=-K_Y\cdot A/A^2$.

\begin{lemma} Suppose $(Y,A)$ is a pseudo-polarised surface such that $Y$ is smooth or toric. If $Y$ is not toric, assume that $-K_Y$ is effective. When $k>\frac{1}{8}\left(\frac{(K_Y\cdot A)^2}{A^2}-K_Y^2\right)$ we have
$$c_k^+(Y,A)=\frac{1}{2}K_Y\cdot A+\sqrt{K_Y^2A^2+2A^2k}$$
\end{lemma}

\begin{proof} Without loss of generality, we assume that $Y$ is smooth. From the Hodge index theorem we have an orthogonal basis $A,e_1,\dots,e_s$ of $\on{NS}(Y)$. Set $e_i^2=-r_i$. Let $-K_Y=\kappa A+\sum\delta_ie_i$. We see that an optimiser for $c_k^+(Y,A)$ is
$$D_k=a_kA-\sum\frac{\delta_i}{2}e_i$$
where $a_k$ is the smallest nonnegative real number $a$ such that
$$a(a+\kappa)\geq\frac{1}{A^2}\left(2k-\sum\frac{\delta_i^2r_i}{4}\right)$$
Solving for $a$, we see that the two solutions are
$$-\frac{\kappa}{2}\pm\sqrt{\frac{\kappa^2}{4}-\sum\frac{\delta_i^2r_i}{4A^2}+\frac{2k}{A^2}}$$
We also note that
$$K_Y^2=\kappa^2A^2-\sum\delta_i^2r_i$$
and so the solutions for $a$ can be rewritten as
$$-\frac{\kappa}{2}\pm\sqrt{\frac{K_Y^2}{4A^2}+\frac{2k}{A^2}}$$
There is a unique nonnegative solution given by the larger value of $a$ precisely when
$$\frac{\kappa^2}{4}<\frac{K_Y^2}{4A^2}+\frac{2k}{A^2}$$
or when
$$k>\frac{\kappa^2A^2}{8}-\frac{K_Y^2}{8}=\frac{1}{8}\left(\frac{(K_Y\cdot A)^2}{A^2}-K_Y^2\right)$$
Substituting in the larger value for $a_k$ gives the result.
\end{proof}

Note that $K_{Y_n}^2=K_{Y_0}^2-n$. Hence, we see that
$$\calg_k(Y_n,A_n)\geq c_k^+(Y_n,A_n)=\frac{1}{2}K_{Y_n}\cdot A_n+\sqrt{K_{Y_n}^2A_n^2+2A_n^2k}$$
for all $k>\frac{1}{8}\left(\frac{(K_{Y_n}\cdot A_n)^2}{A_n^2}-K_{Y_n}^2\right)=\frac{1}{8}n+\frac{1}{8}\left(\frac{(K_{Y_n}\cdot A_n)^2}{A_n^2}-K_{Y_0}^2\right)$. For notational convenience we note that
$$\frac{(K_{Y_n}\cdot A_n)^2}{A_n^2}-K_{Y_0}^2\leq\frac{(K_{Y_0}\cdot A_0)^2}{\mA^2}-K_{Y_0}^2=\frac{9c^2}{c^2-\sum a_i^2}-K_{Y_0}^2=:N$$
We choose a sequence $n_k$ as in Lemma \ref{lem:analysis} with $n_k=o(\sqrt{k})$ and $A_{n_k}^2-\mA^2=o(1/\sqrt{k})$. For sufficiently large $k$ we have $\frac{1}{8}n_k+N<k$. Then, for all such $k$ we have
\begin{align*}
\calg_k(Y_{n_k},A_{n_k})&\geq\frac{1}{2}K_{Y_{n_k}}\cdot A_{n_k}+\sqrt{K_{Y_{n_k}}^2A_{n_k}^2+2A_{n_k}^2k} \\
&=\frac{1}{2}K_{Y_{n_k}}\cdot A_{n_k}+\sqrt{(K_{Y_0}^2-n_k)\left(\mA^2+o\left(\frac{1}{\sqrt{k}}\right)\right)+2\left(\mA^2+o\left(\frac{1}{\sqrt{k}}\right)\right)k} \\
&=\frac{1}{2}K_{Y_{n_k}}\cdot A_{n_k}+\sqrt{(K_{Y_0}^2+o(\sqrt{k}))\left(\mA^2+o\left(\frac{1}{\sqrt{k}}\right)\right)+2\left(\mA^2+o\left(\frac{1}{\sqrt{k}}\right)\right)k} \\
&=\frac{1}{2}K_{Y_{n_k}}\cdot A_{n_k}+\sqrt{2\mA^2k+o(\sqrt{k})}
\end{align*}
As a result, letting $k\to\infty$ gives
\begin{align*}
\liminf_{k\to\infty}\ealg_k(\mY,\mA)&\geq\lim_{k\to\infty}\frac{1}{2}K_{Y_{n_k}}\cdot A_{n_k}+\sqrt{2\mA^2k+o(\sqrt{k})}-\sqrt{2\mA^2k}=\frac{1}{2}K_\mY\cdot\mA
\end{align*}

\begin{prop} Let $(\mY,\mA)$ be a pseudo-polarised Looijenga tower with $\mY$ either smooth or toric. Then
$$\liminf_{k\to\infty}\ealg_k(\mY,\mA)\geq\frac{1}{2}K_\mY\cdot\mA$$
\end{prop}

This implies the following in combinatorial terms.

\begin{cor} Suppose $X_\Omega$ is a convex toric domain. Then
$$\liminf_{k\to\infty}e_k(X_\Omega)\geq-\frac{1}{2}\left(a(\Omega)+b(\Omega)+\laff(\partial^+\Omega)\right)$$
In particular, if $\Omega$ has no rational-sloped edge then
$$\liminf_{k\to\infty}e_k(X_\Omega)\geq-\frac{1}{2}\left(a(\Omega)+b(\Omega)\right)$$
\end{cor}

This completes the proof of Thm.~\ref{thm:loo_main}.

\subsection{Concave toric domains}

We deduce the analogue of Cor.~\ref{cor:main} for concave domains by using a formal property of toric ECH. The formal property in question is described by the following.

\begin{prop}[{\cite[Thm.~A.1]{cri_sym_19}}] Suppose $\Omega$ is a convex toric domain with weight sequence given by $\on{wt}(\Omega)=(c;\on{wt}(\Omega_2),\on{wt}(\Omega_3))$ where $\Omega_2,\Omega_3$ are concave domains as in Def.~\ref{def:conv_wt}. Then
$$\cech_k(X_\Omega)=\inf_{k_2,k_3\geq0}\{\cech_{k+k_2+k_3}(B(c))-\cech_{k_2}(X_{\Omega_2})-c_{k_3}(X_{\Omega_3})\}$$
\end{prop}

Let $\Delta$ be a concave toric domain. It is clear that there exists a convex domain $\Omega$ such that either $\Omega_2=\Delta$ and $\Omega_3=\emptyset$, or $\Omega_2=\emptyset$ and $\Omega_3=\Delta$. We assume the former without loss of generality.

\begin{prop} \label{thm:main_conc} Let $X_\Delta$ be a concave toric domain. Then
$$\liminf_{k\to\infty}e_k(X_\Delta)\geq-\frac{1}{2}(a(\Delta)+b(\Delta)+\laff(\partial^+\Delta))$$
\end{prop}

\begin{proof} Let $\Omega$ be as discussed above and let $c$ be the head of $\on{wt}(\Omega)$. Then
$$\cech_k(X_\Omega)=\inf_{k_2\geq0}\{\cech_k(B(c))-\cech_{k_2}(X_\Delta)\}$$
This infimum is attained for each $k$; we denote an optimiser for $k$ by $k'$ so that
$$\cech_k(X_\Omega)=\cech_k(B(c))-\cech_{k'}(X_\Delta)$$
Thus $e_{k'}(X_\Delta)$ is given by
\begin{align*}
&\cech_{k+k'}(B(c))-\cech_k(X_\Omega)-\sqrt{4\on{vol}(X_\Delta)k'} \\
&=e_{k+k'}(B(c))-e_k(X_\Omega)+\sqrt{4(\on{vol}(X_\Omega)+\on{vol}(X_\Delta))(k+k')}-\sqrt{4\on{vol}(X_\Omega)k}-\sqrt{4\on{vol}(X_\Delta)k'}
\end{align*}
From Cor.~\ref{cor:main} we see that $e_k(X_\Omega)$ and $e_k(B(c))$ are bounded and so it follows that $e_{k'}(X_\Delta)$ is bounded below by
$$-\frac{3c}{2}+\frac{1}{2}(a(\Omega)+b(\Omega)+\laff(\partial^+\Omega))=-\frac{1}{2}(a(\Delta)+b(\Delta)+\laff(\partial^+\Delta))$$
using the Cauchy--Schwartz inequality.
\end{proof}

Hutchings shows in \cite[Cor.~3.9]{hut_ech_19} that $e_k(X_\Delta)$ is bounded above by $-\frac{1}{2}\sum_{a\in\on{wt}(\Delta)}a$. We hence obtain the following.

\begin{thm} \label{cor:main_conc} Let $X_\Delta$ be a concave toric domain. Then
\begin{align*}
-\frac{1}{2}(a(\Delta)+b(\Delta)-\laff(\partial^+\Delta))&\geq\limsup_{k\to\infty}e_k(X_\Delta) \\
&\geq\liminf_{k\to\infty}e_k(X_\Delta)\geq-\frac{1}{2}(a(\Delta)+b(\Delta)+\laff(\partial^+\Delta))
\end{align*}
and so $e_k(X_\Delta)=O(1)$. If $\partial^+\Delta$ has no rational-sloped edges then
$$\lim_{k\to\infty}e_k(X_\Delta)=-\frac{1}{2}(a(\Delta)+b(\Delta))$$
\end{thm}

\begin{proof} The bounds follow immediately from Thm.~\ref{thm:main_conc} and \cite[Cor.~3.9]{hut_ech_19} in combination with\cite[Lem.~3.6]{hut_ech_19}. From here convergence is clear when $\partial^+\Delta$ has no rational-sloped edges.
\end{proof}

\subsection{Algebraic analogues of rational-sloped edges}

We discuss the geometric analogue for polarised Looijenga towers of the combinatorial condition on convex domains of having a rational-sloped edge. In particular, this supplies a criterion for convergence for $\ealg_k(\mY,\mA)$ in this generality. Given a poset $\mP$ define its \emph{extended poset} $\wh{\mP}$ to be $\mP\cup\{\infty\}$ with $\infty>p$ for all $p\in\mP$. If $(\mY,\mA)$ is a pseudo-polarised Looijenga tower we can define an weight function on the extended poset $\wh{\mP}_{(Y_0,L_0)}$ by setting $\on{wt}(\infty)=-K_{Y_0}^+\cdot A_0$. Define a subposet $\wh{\mP}_{(Y_0,L_0)}(p)$ as follows:
\begin{itemize}
\item $p$ is the unique maximal element of $\mP_{(Y_0,L_0)}(p)$,
\item if $q\in\wh{\mP}_{(Y_0,L_0)}(p)$ then exactly one direct descendant of $q$ is in $\wh{\mP}_{(Y_0,L_0)}(p)$, namely the direct descendant corresponding to the point of intersection of $E_q$ and the strict transform of $E_p$ in $Y_q$. 
\end{itemize}

This all works similarly for the weighted poset $\wh{\mP}_\Omega$ associated to a convex domain $\Omega$; for instance, the weight of the element $\infty$ is the affine length of the possibly empty edge of slope $(1,-1)$ in $\partial^+\Omega$, and one can interpret each element $q$ of $\wh{\mP}_\Omega(p)$ with direct ancestor $q'$ as the vertex of $\Omega_{q'}$ incident to the edge that is the moment image of (the strict transform of) $E_p$.

\begin{lemma} Let $X_\Omega$ be a convex toric domain. Let $\wh{\mP}_\Omega$ be the extended weighted poset associated to $\Omega$. Then there is a bijection
$$\text{rational-sloped edges in $\partial^+\Omega$}\longleftrightarrow\text{$p\in\wh{\mP}_\Omega$ such that $\on{wt}(p)-\sum_{q\in\mP_\Omega(p)}\on{wt}(q)>0$}$$
\end{lemma}

\begin{proof} It follows from the weight sequence recursion and the construction of $\mP_\Omega(p)$ that $\on{wt}(p)-\sum_{q\in\mP_\Omega(p)}\on{wt}(q)$ is the affine length of the (possibly empty) edge in $\partial^+\Omega$ introduced at the step corresponding to $p$ in the recursion. Rational-sloped edges in $\partial^+\Omega$ are exactly such edges that have nonzero affine length, which gives the result.
\end{proof}

We see that the extension of $\mP_\Omega$ was necessary to capture the (possiby empty) edge of slope $(1,-1)$ from the first step of the recursion.

\begin{definition} We say that a pseudo-polarised Looijenga tower $(\mY,\mA)$ is \emph{balanced} if $\on{wt}(p)-\sum_{q\in\mP_{(Y_0,L_0)}(p)}\on{wt}(q)=0$ for all $p\in\wh{\mP}_{(Y_0,L_0)}$.
\end{definition}

This is the algebraic analogue for pseudo-polarised Looijenga towers of having no rational-sloped edges in the case of convex domains.

\begin{prop} \label{prop:bal_conv} Suppose $(\mY,\mA)$ is a pseudo-polarised Looijenga tower that is balanced. Then $\ealg_k(\mY,\mA)$ is convergent with
$$\lim_{k\to\infty}\ealg_k(\mY,\mA)=\frac{1}{2}K_\mY\cdot\mA=\frac{1}{2}(K_{Y_0}-K_{Y_0}^+)\cdot A_0$$
\end{prop}

\begin{proof} We already have
$$\frac{1}{2}K_\mY\cdot\mA-K_\mY^+\cdot\mA\geq\limsup_{k\to\infty}\ealg_k(\mY,\mA)\geq\liminf_{k\to\infty}\ealg_k(\mY,\mA)\geq\frac{1}{2}K_\mY\cdot\mA$$
from Thm.~\ref{thm:loo_main}, and so it suffices to show that $-K_\mY^+\cdot\mA=0$ when $(\mY,\mA)$ is balanced. We have
\begin{align*}
-K_\mY^+\cdot\mA&=-K_\mY\cdot\mA+K_{Y_0}\cdot\mA-K_{Y_0}^+\cdot\mA \\
&=-K_{Y_0}\cdot A_0-\sum_{p\in\wh{\mP}_{(Y_0,L_0)}}\on{wt}(p)+K_{Y_0}\cdot A_0-K_{Y_0}^+\cdot A_0 \\
&=\on{wt}(\infty)-\sum_{p\in\wh{\mP}_{(Y_0,L_0)}}\on{wt}(p)
\end{align*}
Let $\mathcal{S}(0)=\{\infty\}$. Recursively define $\mathcal{S}(n)$ to be the set of maxima of
$$\wh{\mP}_{(Y_0,L_0)}\setminus\bigcup_{m< n}\bigcup_{p\in\mathcal{S}(m)}\wh{\mP}_{(Y_0,L_0)}(p)$$
By construction we have from the above that
$$-K_\mY^+\cdot\mA=\sum_{n=0}^\infty\sum_{p\in\mathcal{S}(n)}\left(\on{wt}(p)-\sum_{q\in\wh{\mP}_{(Y_0,L_0)}(p)}\on{wt}(q)\right)$$
which is zero by the assumption that $(\mY,\mA)$ is balanced. The second equality in the statement follows from $-K_\mY^+\cdot\mA=0$.
\end{proof}

\subsection{Outlook} \label{sec:outlook}

We conclude with a selection of ideas and observations that we hope will lead to stronger criteria for convergence or, if one is even expressible, a complete description of what `generic' means in Hutchings' conjecture \cite[Conj.~1.5]{hut_ech_19}.

Given convex or concave $\Omega$ we let $V(\Omega)$ be the $\Q$-vector subspace of $\R$ spanned by the affine lengths of rational-sloped edges in $\partial^+\Omega$. We denote the dimension of $V(\Omega)$ by $v(\Omega)$.

Let $X_\Omega$ be a convex or concave toric domain. We believe that two ingredients for stronger convergence criteria are this $v(\Omega)$ and the number $N(\Omega)$ of rational-sloped edges in $\partial^+\Omega$.

If $N(\Omega)<\infty$ then we suspect $e_k(X_\Omega)$ converges if $v(\Omega)\not=1$. If $\Omega$ has infinitely many rational-sloped edges then it seems likely that $e_k(X_\Omega)$ converges. In each case of convergence we expect that the limit is
\begin{equation} \tag{$\ast$} \label{eqn:ruelle}
-\frac{1}{2}\on{Ru}(X_\Omega)=-\frac{1}{2}(a(\Omega)+b(\Omega))
\end{equation}
though it is possible that there are toric domains for which $e_k$ converge but that are not generic in the sense that they do not satisfy Hutchings' conjecture and have limit different to (\ref{eqn:ruelle}). In the case of non-convergence, we expect that (\ref{eqn:ruelle}) is the midpoint of the lim inf and lim sup of $e_k(X_\Omega)$. Note that the case $v(\Omega)=0$ corresponds to $\Omega$ having no rational-sloped edges -- which is covered by Cor.~\ref{cor:main} and Thm.~\ref{cor:main_conc} -- and $v(\Omega)>1$ corresponds to $\Omega$ having at least two rational-sloped edges whose affine lengths are independent over $\Q$.

There is a distinction between the case that $\Omega$ is a of scaled-lattice type as in \cite{wor_alg_20} -- that is, where $\Omega=q\Omega_0$ for some lattice polygon $\Omega_0$ and some $q\in\R_{>0}$ -- and the complementary case: where either $\Omega$ is polytopal and has $v(\Omega)>1$, or $\Omega$ is not polytopal. In either of the latter situations we have
$$\lim_{k\to\infty}e_{k+1}(X_\Omega)-e_k(X_\Omega)=0$$
and so the obstruction $\limsup_{k\to\infty}e_{k+1}(X_\Omega)-e_k(X_\Omega)>0$ found in the situation of \cite{wor_alg_20} will not assist us in detecting convergence. It is plausible that the case where $N(\Omega)<\infty$ and $v(\Omega)=1$ behaves similarly to to the situation of \cite{wor_alg_20} and has non-convergent $e_k$.

\bibliographystyle{acm}
\bibliography{bw}

\end{document}